\documentclass{amsart}
\usepackage{amsmath,amsthm,amssymb,hyperref,mathrsfs}
\usepackage{aliascnt}
\usepackage{lmodern}

\usepackage{xcolor}
\definecolor{dblue}{rgb}{0,0,0.70}
\hypersetup{
	unicode=true,
	colorlinks=true,
	citecolor=dblue,
	linkcolor=dblue,
	anchorcolor=dblue
}

\makeatletter
\expandafter\g@addto@macro\csname th@plain\endcsname{%
		\thm@notefont{\bfseries}
	}%
\expandafter\g@addto@macro\csname th@remark\endcsname{%
		\thm@headfont{\scshape}
	}%
\makeatother

\newtheorem{theorem}{Theorem}[section]	
\newtheorem*{theorem*}{Theorem}

\newaliascnt{lemma}{theorem}
\newtheorem{lemma}[lemma]{Lemma}
\aliascntresetthe{lemma}

\newaliascnt{proposition}{theorem}
\newtheorem{proposition}[proposition]{Proposition}
\aliascntresetthe{proposition}

\newaliascnt{corollary}{theorem}
\newtheorem{corollary}[corollary]{Corollary}
\aliascntresetthe{corollary}

\newaliascnt{fact}{theorem}
\newtheorem{fact}[fact]{Fact}
\aliascntresetthe{fact}

\theoremstyle{remark}

\newtheorem*{remark}{Remark}

\theoremstyle{definition}
\newtheorem{definition}[theorem]{Definition}

\newaliascnt{example}{theorem}

\aliascntresetthe{example}

\renewcommand{\restriction}{\mathbin\upharpoonright}

\newcommand{\axiom}[1]{\mathsf{#1}} 
\newcommand{\ZFC}{\axiom{ZFC}}
\newcommand{\AC}{\axiom{AC}}
\newcommand{\DC}{\axiom{DC}}
\newcommand{\ZF}{\axiom{ZF}}

\newcommand{\Ord}{\mathrm{Ord}}

\newcommand{\GCH}{\axiom{GCH}}

\newcommand{\HS}{\axiom{HS}}

\DeclareMathOperator{\dom}{dom}
\DeclareMathOperator{\rng}{rng}

\DeclareMathOperator{\sym}{sym}
\DeclareMathOperator{\fix}{fix}

\DeclareMathOperator{\id}{id}
\DeclareMathOperator{\aut}{Aut}

\DeclareMathOperator{\club}{Club}

\newcommand{\forces}{\mathrel{\Vdash}}

\newcommand{\PP}{\mathbb{P}}
\newcommand{\power}{\mathcal{P}}

\newcommand{\QQ}{\mathbb{Q}}
\newcommand{\RR}{\mathbb{R}}

\newcommand{\cF}{\mathcal F}
\newcommand{\sF}{\mathscr F}
\newcommand{\cG}{\mathcal G}
\newcommand{\sG}{\mathscr G}

\newcommand{\tup}[1]{\langle#1\rangle}

\author{Asaf Karagila}
\thanks{This paper is part of the author's Ph.D.\ written at the Hebrew University of Jerusalem under the supervision of Prof.\ Menachem Magidor.}
\address{Einstein Institute of Mathematics.
Edmond J. Safra Campus, Givat Ram,
The Hebrew University of Jerusalem. 
Jerusalem, 91904, Israel}
\email{karagila@math.huji.ac.il}
\urladdr{http://boolesrings.org/asafk}
\date{\today}

\title{Fodor's lemma can fail everywhere}

\begin{document}
\begin{abstract}We show that it is equiconsistent with $\ZF$ that Fodor's lemma fails everywhere, and furthermore that the club filter on every regular cardinal is not even $\sigma$-complete. Moreover, these failures can be controlled in a very precise manner.
\end{abstract}
\maketitle    
\section{Introduction}
Closed and unbounded sets (clubs) and stationary sets are central to modern set theory. They give us a good notion of how ubiquitous is some property, or how large is a set. One lemma which is more central than others is Fodor's lemma, which to some extent can be seen as a generalization of the pigeonhole principle, which asserts that if $\kappa$ is a regular uncountable cardinal and $f\colon\kappa\to\kappa$ satisfies $f(\alpha)<\alpha$ for all $\alpha>0$, then $f$ is constant on a set of size $\kappa$ (which is also stationary). The fact that the proof of Fodor's lemma requires us to use the axiom of choice is not the least surprising, but it is curious nonetheless. This is an unwritten corollary in \cite{Bilinsky-Gitik:2012} where Eilon Bilinsky and Moti Gitik construct a model where a measurable cardinal has no normal measures. Here we take a different approach to the problem, and get that an arbitrary regular cardinal can be without normal filters, as well as global generalizations of this.

In this work we investigate some abstract properties related to Fodor's lemma and show that they can fail without choice, and they can hold with even the axiom of countable choice failing. Starting with any model $V$ of $\ZFC+\GCH$, we construct a model of $\ZF$ with the same cofinalities as $V$, in which Fodor's lemma fails for any prescribed class of regular cardinals. Moreover, in this extension we can also control the completeness of the club filters of every regular cardinal. It is worth pointing out explicitly that these assumptions mean that every successor cardinal is regular. This is worth noting, as any universal statement about Fodor's lemma and club filters will vacuously hold true in Gitik's infamous model, constructed in \cite{Gitik:1980}, where all limit ordinals have countable cofinality. So our proofs have substance to them, and they do not require any large cardinals.

\subsection{The structure of this paper}
Section~\ref{sect:prelim} covers the basics of symmetric extensions, which is the main tool used to create models where the axiom of choice fails; section~\ref{sect:fodor} includes some proofs in $\ZF$ about Fodor's lemma and the completeness of the club filter; and in sections \ref{sect:local} and \ref{sect:global} we construct the models, first a localized failure at a single regular cardinal and then two methods are outlined to construct models of global failure.

\section{Some preliminaries}\label{sect:prelim}
In this work a notion of forcing is a preordered set with a maximum element denoted by $1$. If $\PP$ is a notion of forcing and $p,q\in\PP$, we say that $q$ is \textit{stronger} than $p$ if $q\leq p$, alternatively we might say that $q$ extends $p$. If $\PP$ is a forcing, we will always use $\dot x$ to denote a $\PP$-name, and $\check x$ to denote the canonical name for $x$ in the ground model. We follow the convention that $\kappa$-distributive and $\kappa$-closed refer to relevant sequences \textit{shorter} than $\kappa$ having lower bounds.

Given an ordinal $\lambda$ we denote by $\AC_\lambda$ the statement that whenever $\{A_\alpha\mid\alpha<\lambda\}$ is a family of non-empty sets, there is a function such that $f(\alpha)\in A_\alpha$ for all $\alpha<\lambda$. We shall denote by $\DC_\lambda$ the statement that whenever $T$ is a tree of height $\leq\lambda$ which is $\lambda$-closed and has no maximal nodes, then there is a cofinal branch in $T$. In both cases we write $\AC_{<\lambda}$ and $\DC_{<\lambda}$ to abbreviate $(\forall\kappa<\lambda)\AC_\kappa$ and $(\forall\kappa<\lambda)\DC_\kappa$ respectively. More on the connection between the two statements can be found in \cite[Th.~8.1]{Jech:AC1973}.

Symmetric extensions are inner models of generic extensions where the axiom of choice can fail. Let us review the basic definitions related to symmetric extensions. Suppose that $\PP$ is a forcing notion and $\pi$ is an automorphism of $\PP$, then $\pi$ extends to a permutation of the $\PP$-names by recursion on the rank of $\dot x$:\[\pi\dot x=\{\tup{\pi p,\pi\dot y}\mid \tup{p,\dot y}\in\dot x\}.\]
Suppose that $\sG\leq\aut(\PP)$ is an automorphism group of $\PP$, we denote by $\sym_\sG(\dot x)$ the group $\{\pi\in\sG\mid\pi\dot x=\dot x\}$. We say that \textit{$\sG$ witnesses the homogeneity of $\PP$} if whenever $p,q\in\PP$ there is some $\pi\in\sG$ such that $\pi p$ is compatible with $q$. $\PP$ is \textit{weakly homogeneous} when $\aut(\PP)$ witnesses the homogeneity of $\PP$.

We say that $\sF$ is a \textit{filter of subgroups} over $\sG$ if it is a filter with respect to the lattice of subgroups of $\sG$. Namely, we take a filter (in the usual sense) of subsets of $\sG$ and consider the groups generated by each set in the filter. We say that $\sF$ is \textit{normal} if it is closed under conjugation, so if $H\in\sF$ and $\pi\in\sG$, then $\pi H\pi^{-1}\in\sF$ as well.\footnote{We are aware of the unfortunate overlap of the term ``normal filter''. The terminology, however, is standard in the context of symmetric extensions, as evident in \cite{Jech:AC1973}. It will always be clear when we refer to this notion of normality.}  A \textit{symmetric system} is $\tup{\PP,\sG,\sF}$ such that $\sG$ is a group of automorphisms of $\PP$ and $\sF$ is a normal filter of subgroups over $\sG$. 

We say that a $\PP$-name $\dot x$ is $\sF$-symmetric if $\sym_\sG(\dot x)\in\sF$; and we say that $\dot x$ is hereditarily $\sF$-symmetric if in addition to being $\sF$-symmetric, every $\dot y$ which appears in $\dot x$ is already hereditarily $\sF$-symmetric. We will denote by $\HS_\sF$ the class of all hereditarily $\sF$-symmetric names. And from this point on we shall omit $\sG$ and $\sF$ when they are clear from context.

\begin{lemma}[The Symmetry Lemma]
Suppose that $\pi\in\aut(\PP)$ then for every $p,\dot x$ and $\varphi$,
\[p\forces\varphi(\dot x)\iff\pi p\forces\varphi(\pi\dot x).\qed\]
\end{lemma}
\begin{theorem}
Suppose that $\tup{\PP,\sG,\sF}$ is a symmetric system and $G$ is a $V$-generic filter for $\PP$. Then $M=\HS^G=\{\dot x^G\mid\dot x\in\HS\}$ is a transitive class of $V[G]$ satisfying $\ZF$ and $V\subseteq M$.\qed
\end{theorem}
The model $M$ in the above theorem is called a \textit{symmetric extension} of $V$. The proof of the above statements, as well a much more detailed discussion of symmetric extensions can be found in \cite{Jech:AC1973}. We will also take products of symmetric systems, by this we mean that we take some type of product of the forcings, and using the same support we take products of the automorphism groups, and these act on the product forcing pointwise.

We will also need the following general fact, whose proof is due to Yair Hayut.
\begin{lemma}\label{lemma:yair product}
Let $\kappa$ be a regular cardinal. Suppose that $\PP$ has $\kappa$-c.c.\ and $\QQ$ is a $\kappa$-distributive forcing such that $\forces_\QQ``\check\PP\text{ has }\check\kappa\text{-c.c.}"$. Then $\forces_\PP``\check\QQ\text{ is }\kappa\text{-distributive}"$.
\end{lemma}
\begin{proof}
Let $\dot A$ be a $\PP\times\QQ$-name for a set of ordinals such that $\forces_{\PP\times\QQ}|\dot A|<\check\kappa$. Let $G$ be a $V$-generic filter for $\QQ$, working in $V[G]$ we still have that $\PP$ has $\kappa$-c.c.\ and $\kappa$ is still a regular cardinal. Let $\dot A^G$ denote the $\PP$-name obtained from $\dot A$ and $G$; there is some $\PP$-name $\dot B$ such that $|\dot B|<\kappa$ and $\forces_\PP\dot A^G=\dot B$. Using the fact that $\QQ$ is $\kappa$-distributive we obtain that $\dot B\in V$, and therefore there is some $q\in G$ such that $(1_\PP,q)\forces_{\PP\times\QQ}\dot A=\dot B$. Thus, every name for a set of ordinals smaller than $\kappa$ can be realized as a $\PP$-name in $V$ already, so $\forces_\PP``\check\QQ\text{ is }\kappa\text{-distributive}"$ as wanted.
\end{proof}

\section{The Fodor property and some theorems in \texorpdfstring{$\ZF$}{ZF}}\label{sect:fodor}
\begin{theorem}
Suppose that $\lambda$ is an uncountable regular cardinal, $\{S_\alpha\mid\alpha<\mu\}$ is a sequence of clubs of $\lambda$. If $\mu<\lambda$, then $\bigcap\{S_\alpha\mid\alpha<\mu\}$ is a club; and if $\mu=\lambda$ then the diagonal intersection $\triangle\{S_\alpha\mid\alpha<\mu\}$ is a club.
\end{theorem}
\begin{proof}
The usual proof in $\ZFC$ can be repeated, or we can use an absoluteness trick. First note that $S=\{\tup{\alpha,\beta}\mid\beta\in S_\alpha\}\subseteq L$, and therefore in $L[S]\subseteq V$ we have the sequence of the clubs; $\lambda$ is a regular cardinal; and the axiom of choice holds true. Therefore the intersection (or diagonal intersection) is a club in $L[S]$. Being closed and unbounded in $\lambda$ is absolute between transitive models, and therefore the same holds in $V$.
\end{proof}
\begin{definition}
Let $\lambda$ be an uncountable regular cardinal and $\mu\leq\lambda$. We say that $\lambda$ has the \textit{$(\lambda,\mu)$-Fodor property} if whenever $S\subseteq\lambda$ is a stationary set and $f\colon S\to\lambda$ is a regressive function with range of cardinality $\leq\mu$, there is some $\alpha$ such that $f^{-1}(\{\alpha\})$ is stationary.
\end{definition}
\begin{theorem}
Let $\lambda$ be an uncountable regular cardinal. The following are equivalent for any $\mu<\lambda$:
\begin{enumerate}
\item $\lambda$ has the $(\lambda,\mu)$-Fodor property.
\item The club filter of $\lambda$ is $\mu^+$-complete.
\end{enumerate}
Furthermore, if $\mu=\lambda$, then $(\lambda,\lambda)$-Fodor property holds if and only if the club filter of $\lambda$ is closed under diagonal intersection of $\lambda$-sequences, i.e.\ the club filter is normal.
\end{theorem}
\begin{proof}
Assume $\mu<\lambda$ and the $(\lambda,\mu)$-Fodor property holds for $\lambda$. Let $\{S_\alpha\mid\alpha<\mu\}$ be such that each $S_\alpha$ contains a club. If $S=\bigcap\{S_\alpha\mid\alpha<\mu\}$ does not contain a club, then $T=\lambda\setminus S$ is stationary. Define a function $f$ on $T\setminus\mu$, $f(\xi)=\alpha$ if and only if $\alpha$ is the least such that $\xi\notin S_\alpha$. Clearly $f$ is regressive on a stationary set, so there is some $\alpha$ such that $f^{-1}(\{\alpha\})$ is stationary. But then $f^{-1}(\{\alpha\})\cap S_\alpha=\varnothing$, which is impossible since $S_\alpha$ contains a club. Therefore $T$ must be non-stationary, so $S$ contains a club.

In the other direction, suppose that $f\colon S\to\lambda$ is a regressive function with range of cardinality $\leq\mu$. We may assume that $\rng(f)\subseteq\mu$ by composing with a suitable bijection. If there is no $\alpha<\mu$ such that $f^{-1}(\{\alpha\})$ is stationary, then $T_\alpha=\lambda\setminus f^{-1}(\{\alpha\})$ contains a club for all $\alpha$, and therefore $T=\bigcap\{T_\alpha\mid\alpha<\mu\}$ contains a club. But this is impossible as for $\xi\in S\cap T$ it follows that $f(\xi)$ is undefined, and such $\xi$ exists because $S$ is stationary.

For the case that $\mu=\lambda$ the same proof works, replacing the intersection by a diagonal intersection where needed.
\end{proof}

\begin{theorem}
Assume that $\lambda$ is an uncountable regular cardinal, and $\mu\leq\lambda$ is regular such that $\AC_\mu$ holds. Then the $(\lambda,\mu)$-Fodor property holds.
\end{theorem}
\begin{proof}
By the previous theorem, it is enough to prove that the club filter is $\mu^+$-complete, or closed under diagonal intersections. And indeed, if $\{S_\alpha\mid\alpha<\mu\}$ is such that $S_\alpha$ contains a club for each $\alpha$, then by $\AC_\mu$ we can choose such club $T_\alpha$, for every $\alpha$. The (diagonal) intersection of the $T_\alpha$'s witnesses the wanted closure.
\end{proof}
\begin{theorem}
The statement ``For every uncountable regular $\lambda$, the $(\lambda,\lambda)$-Fodor property holds'' does not imply $\AC_\omega$.
\end{theorem}
\begin{proof}
Gitik has proved the consistency of ``every limit ordinal has cofinality $\omega$'' in \cite{Gitik:1980}. In such model, the statement holds vacuously. However this is a serious cop out from an actual proof, so we will instead provide an outline of the consistency of ``For every uncountable regular $\lambda$, the $(\lambda,\lambda)$-Fodor property holds'' along with ``Every successor cardinal is regular'' and the failure of $\AC_\omega$.

Suppose that $V$ is a model of $\ZFC$ and $V[G]$ is a generic extension of $V$ obtained by adding Cohen reals. If $M$ is a model of $\ZF$ which lies between $V$ and $V[G]$, then $M$, $V$ and $V[G]$ agree on cofinality of the ordinals, and therefore agree on which ordinals are cardinals. In particular every successor cardinal in $M$ is regular there.

Suppose now that $\lambda$ is regular and uncountable, and $\{S_\alpha\mid\alpha<\lambda\}\in M$ is a family of sets which contain clubs. Then $\vec S=\{\tup{\alpha,\beta}\mid\beta\in S_\alpha\}$ is a subset of $V$. It follows that $V[S]$ is a model of $\ZFC$ between $V$ and $V[G]$, and all three models agree on the notion of clubs because Cohen forcing is c.c.c. Therefore the diagonal intersection of the $S_\alpha$'s, which lie in $V[S]$ contains a club there, which is also in $M$. Therefore $(\lambda,\lambda)$-Fodor holds.

The models which lie between $V$ and even an extension by a single Cohen real are many. And there are plenty of which where $\AC_\omega$ fails. For example, Cohen's first model, which is the standard model in which the reals cannot be well-ordered (details can be found in \cite[\S5.3]{Jech:AC1973}).
\end{proof}
\section{Destroying the completeness of the club filter at \texorpdfstring{$\kappa$}{k}}\label{sect:local}
In \cite{Bilinsky-Gitik:2012} Eilon Bilinsky and Moti Gitik constructed a model in which a measurable cardinal carries no normal measures. The paper begins by constructing a model in which the formerly-measurable cardinal carries no normal filters at all. In such model, of course, the Fodor property fails for that cardinal. The approached taken there is suited for inaccessible cardinals. Here consider a slightly different approach which is somewhat more flexible, and can be later used for the global failure.

Let $\kappa$ be an uncountable regular cardinal, and suppose that $\mu\leq\kappa$ is a regular cardinal. We will construct a model in which the $(\kappa,\mu)$-Fodor property fails, but $\DC_{<\mu}$ holds. 

\begin{lemma}
Suppose that $\lambda$ is regular and $S\subseteq\lambda$ is a stationary set. Let $\club(S)$ be the forcing for shooting a club through $S$ using closed and bounded subsets of $S$ as conditions. Then $\club(S)$ is weakly homogeneous.
\end{lemma}
\begin{proof}
By \cite[\S3.5,~Th.~1]{Grigorieff:1975} it suffices to show that if $G$ is $V$-generic for $\club(S)$, then for every $p\in\club(S)$ there is some $V$-generic $H$ with $p\in H$ and $V[G]=V[H]$. This is trivial, since if $G$ is the generic club, we can change any bounded part of $G$ to accommodate any condition as an initial segment while preserving $V$-genericity.
\end{proof}

Our forcing is given by a two-step iteration. The first step is to add a Cohen subset to $\kappa$, considering this subset as a function from $\kappa$ to $\mu$, this gives us a sequence of $\mu$ stationary sets, $S_\alpha$, with an empty intersection. We will shoot clubs into each of these stationary sets by taking a ${<}\mu$-support product of $\club(S_\alpha)$. We will argue that we have sufficient closure to ensure that $\DC_{<\mu}$ holds, but while each $S_\alpha$ contains a club, their intersection is empty.

Let $\QQ_0$ be the forcing whose conditions are partial functions $p\colon\kappa\to\mu$ with $\dom p$ bounded, ordered by inverse inclusion. We will denote by $g$ the generic function added by $\QQ_0$ and $\dot g$ will be its canonical name. Denote by $S_\alpha$ the set $g^{-1}([\alpha,\mu))$, then by easy density arguments $S_\alpha$ is fat for every $\alpha<\mu$. Denote by $\QQ_{1,\alpha}$ the forcing $\club(S_\alpha)$, and $\QQ_1$ is the ${<}\mu$-support product $\prod_{\alpha<\mu}\QQ_{1,\alpha}$. If $E\subseteq\mu$, we will write $\QQ_1\restriction E$ to denote the restriction of the product to $\prod_{\alpha\in E}\QQ_{1,\alpha}$, and if $q$ is a condition in $\QQ_1$, the restriction $q\restriction E$ will denote the projection of $q$ to $\QQ_1\restriction E$. Note that if $E$ is bounded in $\mu$, then $\QQ_1\restriction E$ is in fact a full-support product.

We take $\PP$ to be the two-step iteration $\QQ_0\ast\dot\QQ_1$. Using the technique described in \cite[\S3]{Karagila:2016}, every name $\dot\sigma$ for an automorphism of $\QQ_1$ can be used to define an automorphism $\sigma$ of $\PP$ using the following definition:\[\sigma\tup{p,\dot q}=\tup{p,\dot\sigma(\dot q)}.\footnote{We avoid the integral notation here, since we only use the automorphisms of $\dot\QQ_1$.}\]
Our group of automorphisms will be the group of automorphisms induced from names of automorphisms of $\dot\QQ_1$ which are forced to belong to the ${<}\mu$-support product of the groups $\aut(\QQ_{1,\alpha})$. When $\sigma$ is an automorphism, it will be the one induced by the name $\dot\sigma$, and vice versa.
\begin{remark}
Note that the conditions of $\dot\QQ_1$ are in fact in $V$. Therefore we can give every name of an automorphism a canonical name. Moreover, \cite[Prop.~3.10]{Karagila:2016} shows that if $1\forces_{\QQ_0}\dot\sigma=\dot\sigma'$, then they actually induce the same automorpshim. This defines an equivalence relation on the set of canonical names, and we can choose one from each equivalence class if we wish to do so for defining the group in specific details.
\end{remark}

As the theorem above show, each $\QQ_{1,\alpha}$ is weakly homogeneous. So the product $\QQ_1$ is weakly homogeneous. We define $\sF$ to be the filter of subgroups generated by $\fix(E)$ for $E\in[\mu]^{<\mu}$, where $\sigma\in\fix(E)$ if for every $\alpha\in E$, $\forces_{\QQ_0}\dot\sigma\restriction\dot\QQ_{1,\alpha}=\id$. Note that $\mu\leq\kappa$ and $\QQ_0$ is $\kappa$-closed, therefore $[\kappa]^{<\mu}$ is the same in $V[g]$ and in $V$.

We will prove two lemmas in order to obtain the wanted result. 
\begin{lemma}\label{lemma:approximation}
Suppose that $\dot A\in\HS$ is a $\PP$-name for a set of ordinals. Then there is some $\alpha<\mu$ and $\dot A'\in\HS$ which is a $\QQ_0\ast\dot\QQ_1\restriction\alpha$-name and $1\forces\dot A=\dot A'$.
\end{lemma}
\begin{proof}
We can replace $\dot A$ by the name $\{\tup{p,\check\xi}\mid p\forces\check\xi\in\dot A\}$, by the Symmetry Lemma both are stable under the same automorphisms, so we may assume that every name which appears in $\dot A$ is of the form $\check\xi$ for some ordinal $\xi$. Let $E\in[\mu]^{<\mu}$ such that $\fix(E)\leq\sym(\dot A)$. Suppose that $\tup{p,\dot q}\forces\check\xi\in\dot A$, then by homogeneity we get that for every $\tup{p,\dot q'}$ with $p\forces\dot q'\leq_{\QQ_1}\dot q\restriction\alpha$, there is some $\dot\sigma$ such that $\dot\sigma\in\fix(E)$ and $\dot\sigma(\dot q\restriction(\mu\setminus E))$ is compatible with $\dot q'$. Therefore $\tup{p,\dot q'}\forces\check\xi\in\dot A$.

It follows that $\tup{p,\dot q\restriction E}\forces\check\xi\in\dot A$. By taking $\alpha=\sup E$ we get the wanted result. Namely, $\forces\dot A=\{\tup{\tup{p,\dot q\restriction\alpha},\check\xi}\mid\tup{p,\dot q}\forces\check\xi\in\dot A\}$.
\end{proof}
\begin{lemma}
Suppose that $\alpha<\mu$. Then the two step iteration $\QQ_0\ast\dot\QQ_1\restriction\alpha$ has a $\kappa$-closed dense subset.
\end{lemma}
\begin{proof}
Let $D$ be the set of conditions of the form $\tup{p,\check q}\in\QQ_0\ast\dot\QQ_1\restriction\alpha$ such that there is some $\gamma<\kappa$ such for all $\beta<\alpha$, $\max q(\beta)=\gamma$, and $\dom p=\gamma+1$ with $p(\gamma)>\alpha$.

Suppose that $\tup{p',\dot q'}$ is an arbitrary condition, by extending $p'$ we can find some condition $p\leq_{\QQ_0}p'$ such that the domain of $p$ is some $\gamma+1$, $p(\gamma)\geq\alpha$ and there is some $q'\in V$ such that $\gamma>\max q'(\beta)$ for all $\beta<\alpha$ and $p\forces_{\QQ_0}\check q'=\dot q'$. Now take $q$ to be the one-point extensions of each $q'(\beta)$ by adding $\gamma$ as a new maximal element.

To see that $D$ is $\kappa$-closed, note that if $\{\tup{p_\xi,\check q_\xi}\mid\xi<\eta\}$ for some $\eta<\kappa$ is a decreasing sequence of conditions in $D$. Let $\gamma$ be the $\sup\dom p_\xi$, which is also $\sup\bigcup\{q_\xi(0)\mid\xi<\eta\}$ (we can replace $0$ by any $\beta<\alpha$ by the requirement on the maximal elements of the conditions in $q$), then taking $p=\bigcup\{p_\xi\mid\xi<\eta\}\cup\{\tup{\gamma,\alpha}\}$ is a lower bound to all the $p_\xi$'s, and taking $q(\beta)=\bigcup\{q_\xi(\beta)\mid\xi<\eta\}\cup\{\gamma\}$ is a closed subset of each $S_\beta$ for $\beta<\alpha$ and a lower bound to all the $q_\xi$'s, and $\tup{p,\check q}\in D$.
\end{proof}
\begin{corollary}
$\PP$ has dense subset which is $\mu$-closed.\qed
\end{corollary}
Let $G$ be a $V$-generic filter for $\PP$ and let $M=\HS^G$. We will now prove that $M$ satisfies our desires. Namely, the $(\kappa,\mu)$-Fodor property fails, while $\DC_{<\mu}$ holds and for regular $\lambda\neq\kappa$ the $(\lambda,\lambda)$-Fodor property holds.
\begin{theorem}\label{thm:main thm for one}
In $M$ for every $\alpha<\mu$, $S_\alpha$ contains a club, but the intersection $\bigcap_{\alpha<\mu} S_\alpha=\varnothing$. Therefore in $M$ the $(\kappa,\mu)$-Fodor property fails; $\DC_{<\mu}$ holds; and for every regular cardinal $\lambda$, if $\lambda\neq\kappa$, then the $(\lambda,\lambda)$-Fodor property holds.
\end{theorem}
\begin{proof}
Each $S_\alpha$ contains a club in $M$, since $E=\{\alpha\}$ witnesses that the generic club for $\QQ_{1,\alpha}$ is in $\HS$; and $\bigcap_{\alpha<\mu}S_\alpha=\varnothing$ by the definition of $S_\alpha$ for $\alpha<\mu$. In the case of $\mu=\kappa$, we get that $S_\kappa=\mathop{\triangle}_{\alpha<\kappa}S_\alpha$ is stationary, since it contains a club after forcing with $\PP$ and $\kappa$ is preserved. However, it is still true that in $M$, $S_\kappa$ does not contain a club, since such a club would have to be added by $\QQ_1\restriction\alpha$ for some $\alpha<\kappa$, and this is impossible as that would mean that for every $S_\beta$ for $\beta>\alpha$, we have added a club at that stage. So the club filter is not normal.

We get $\DC_{<\mu}$ by \cite[Lemma~2.1]{Karagila:2014}, since $\PP$ is $\mu$-closed, and since $\sF$ is $\mu$-complete.

If $\lambda<\kappa$, then by the fact every set of ordinals was introduced by a $\kappa$-closed subforcing, we did not add any new subsets to $\lambda$ between $V$ and $V[G]$, and by the fact that $M$ lies between the two models $\power(\lambda)$ is computed the same in $V,M$ and $V[G]$. Therefore the conclusion holds. If $\lambda>\kappa$, then by the fact that $\PP$ satisfies $\kappa^+$-c.c.\ it follows that if $\tup{\dot T_\alpha\mid\alpha<\lambda}^\bullet\in\HS$ is a name for a sequence, and $p$ is a condition such that for every $\alpha<\lambda$ there is some $\dot C_\alpha\in\HS$ such that $p\forces\dot C_\alpha\subseteq\dot T_\alpha\text{ is a club}$, then there is a club $D_\alpha\in V$ such that $p\forces\check D_\alpha\subseteq\dot C_\alpha$, which completes the proof.
\end{proof}
\begin{remark}
Note that in $V[G]$ we collapsed $\kappa$ to be $\mu$, in the case that $\mu<\kappa$, as the fact each $S_\alpha$ contains a club is upwards absolute as well as the fact their intersection is empty. This is somewhat common in choiceless constructions, where in the outer model cardinals were collapsed, but in the symmetric extensions we managed to keep them standing.
\end{remark}
\section{Failure, failure everywhere!}\label{sect:global}
Let $D$ be a class of regular cardinals and let $F$ be a function defined on $D$ such that $F(\kappa)$ is a regular cardinal $\leq\kappa$ for all $\kappa\in D$; and for $\lambda\notin D$, the $(\lambda,\lambda)$-Fodor property holds. We would like to have that $(<F(\kappa),\kappa)$-Fodor property holds while the $(F(\kappa),\kappa)$-Fodor property fails for $\kappa\in D$. 

In this section we outline two methods which can be used to construct such models: Easton products of symmetric systems, and symmetric iterations. The application of iterated symmetric extensions is presented here as a proof of concept, suggesting that the method can be applied to other constructions of similar flavor, where the first method is inapplicable (e.g.\ Gitik's model). We present the two methods, starting with a model $V$ satisfying $\ZFC+\GCH$. 

One minor problem that occur, however, is that both methods require that the localized forcing $\QQ_0\ast\dot\QQ_1$ outlined in the previous section is $\kappa$-closed. This is only true when $\mu=\kappa$, namely when we only destroy the normality of the club filter at $\kappa$. We can fix this problem when $\mu<\kappa$ by paying a small price of elegance: it can no longer be true that the intersection of sets containing clubs is empty. The modification is as follows: The generic function $g\colon\kappa\to\mu$ of $\QQ_0$ is now into $\mu+1$, and the product of $\dot\QQ_1$ is now the ${<}\kappa$-support product of $\club(S_\alpha)$ for $\alpha<\mu$. Namely, we shoot clubs into all the stationary sets, except the last one. The automorphism group is now taken to be the ${<}\kappa$-support product instead of the ${<}\mu$-support product; but the filter of subgroups remains the same.

The proof of \autoref{lemma:approximation} needs no modifications, and therefore the intersection of the generic clubs does not enter the symmetric model, so indeed the intersection of the $S_\alpha$'s does not contain a club. However it is nevertheless stationary as it includes $S_\mu$, the last stationary set. What we do gain here is that the iteration is now $\kappa$-closed. So we can now obtain the global results.

\subsection{Easton product of symmetric systems}
We give a rough outline of the construction and the arguments, a more detailed account can be found in \cite[\S4]{Karagila:2014}.  

For $\kappa\in D$ let $\tup{\PP_\kappa,\sG_\kappa,\sF_\kappa}$ be the symmetric system from the previous section for $F(\kappa)=\kappa$, or the modified version for $F(\kappa)<\kappa$, with $F(\kappa)$ taking the role of $\mu$. We define $\PP$ to be the Easton support product of the $\PP_\kappa$, with $\sG$ defined as the Easton support product of the $\sG_\kappa$, and $\sF$ the Easton support product of the $\sF_\kappa$. Namely $\pi\in\sG$ is a function with a domain which is a set $A\subseteq D$, such that $\pi(\kappa)\in\sG_\kappa$, and the action of $\pi$ on $\PP$ is defined pointwise. We say that $\vec H\in\sF$ if $\vec H$ is a function whose domain is a subset $A\subseteq D$ and $\vec H(\kappa)\in\sF_\kappa$ for all $\kappa\in A$, and we say that that $\pi$ lies in $\vec H$ if for every $\kappa\in\dom\vec H$, $\pi(\kappa)\in \vec H(\kappa)$.

We define $\HS$ as before. Each $\PP_\kappa$ is $\kappa$-closed, therefore the Easton product preserves $\ZFC$ in the outer model, thus we can appeal to the standard arguments that $\HS^G$ is a transitive class, closed under G\"odel operations and almost universal, so it is a model of $\ZF$. By arguments similar to \cite[\S4]{Karagila:2014}, with the proofs in the previous section it follows that $(\kappa,F(\kappa))$-Fodor fails for all $\kappa\in D$ while the $(\kappa,\kappa)$-Fodor property holds, for all $\kappa\notin D$. Moreover $\DC_{<\mu}$ holds, where $\mu=\min\rng F$.
\subsection{Symmetric iterations}
The method of symmetric iterations was developed in \cite{Karagila:2016}. The ideas behind this method, however, relies heavily on the use of a finite support iteration---or in this case, a finite support product. If we take a class product with finite support we are bound to add a proper class of Cohen reals, and we are likely to collapse all the cardinals to be countable. This deprives us of the arguments based on G\"odel operations---as used in the previous part---that symmetric extensions satisfy $\ZF$, as these arguments rely heavily on $\ZF$ holding in the outer model. In constructions such as Gitik's in \cite{Gitik:1980} or Fernengel--Koepke's recent work in \cite{Fernengel-Koepke:2016} the authors argue that $\ZF$ holds in the resulting model by verifying the axioms directly from the construction of the model, an approach for which your miles may vary depending on the complexity of the construction.

Meanwhile, the technique of iterating symmetric extensions offers us a simpler way, both in terms of thinking about these constructions as iterated constructions, where we solve a few problems at a time, as well in the existence of a preservation theorem for particular cases which include some class length iterations and products. The following fact is the application of the preservation theorem to our case. The details can be found in \cite[\S9]{Karagila:2016}.
\begin{fact}
Suppose that $\tup{\QQ_\alpha,\sG_\alpha,\sF_\alpha\mid\alpha\in\Ord}$ is a class of symmetric systems, let $\PP_\alpha$ denote the finite support product of these systems for $\beta<\alpha$ and $\HS_\alpha$ denote the hereditarily symmetric names for these products, finally assume that $G$ is a $V$-generic for $\PP_{\Ord}=\PP$. If for all $\alpha$, $\sG_\alpha$ witnesses the homogeneity of $\QQ_\alpha$, and for all $\eta$ there is some $\rho(\eta)$ such that for all $\alpha\geq \rho(\eta)$, $V[G]_\eta\cap(\HS_\alpha)^G=V[G]_\eta\cap(\HS_{\alpha+1})^G$, then $\HS^G$ satisfies $\ZF$.\footnote{This implies that we somehow require $\rho$ to be definable in $V$, but this is a consequence of the fact that not adding sets of rank $\eta$ is something definable from $\eta$.}
\end{fact}
In other words, under the assumption of homogeneity if we can guarantee for every $\eta$, that at some point no new sets of rank $<\eta$ will be added, then the final model will satisfy $\ZF$. 
\medskip

Our current predicament is not as straightforward as to apply this fact directly. While the forcing $\QQ_0\ast\dot\QQ_1$ described in the previous section is weakly homogeneous, the group of automorphisms we used does not witness that, as no conditions in $\QQ_0$ are moved by these automorphisms.

We could use automorphisms that move conditions in $\QQ_0$ as well, but we will need to ensure that the name for $\QQ_1$ is preserved. Another option would be to use a slightly different forcing for $\QQ_1$ which is more amenable to automorphisms of $\QQ_0$. A third option, which we will pursue for the rest of the section is to force with an Easton product that adds all the Cohen subsets for the $\QQ_0$'s, and in that model take the finite support product of the $\QQ_1$'s.

Let $W$ be the model obtained by forcing with the Easton product of $\QQ_{\kappa,0}$, adding a Cohen function $g_\kappa\colon\kappa\to F(\kappa)$ over $V$ for $\kappa\in D$.

In $W$ define $\QQ_\kappa$ to be $\dot\QQ_{\kappa,1}^{g_\kappa}$, the forcing as interpreted in $V[g_\kappa]$, for any $\alpha\notin D$ we define $\QQ_\alpha$ as the trivial forcing. This gives us, in fact, that the symmetric system of $\PP_\kappa$ from the previous section has a natural interpretation in this model, as the automorphisms were really just defined on $\QQ_{\kappa,1}$, and it is a weakly homogeneous system.
\begin{proposition}
For every $\kappa\in D$, $W$ satisfies that $\QQ_\kappa$ is a $\kappa$-distributive forcing.
\end{proposition}
\begin{proof}
To get from $V$ to $W$ we took an Easton support product of increasingly more closed forcings. Therefore it is enough to work over $W'=V[\tup{g_\mu\mid\mu\in D\cap\kappa+1}]$. Now, we argue that in fact the definition of $\QQ_\kappa$ as $\dot\QQ_{1,\kappa}^{g_\kappa}$, we get that forcing with $\QQ_\kappa$ over $W'$ is the same as forcing over $V$ with the Easton support product of $\left(\prod_{\mu<\kappa}\QQ_{0,\mu}\right)\times\PP_\kappa$. We proved that $\PP_\kappa$ is $\kappa$-closed. So if $[H_0\times H_1]$ is a $V$-generic filter for $\prod_{\mu<\kappa}\QQ_\mu\times\PP_\kappa$, such that $V[H_0]=W'$, we get that $V[H_0\times H_1]=W'[H-1]$ and both have the same $\gamma$-sequences for all $\gamma<\kappa$.
\end{proof}

Let $\RR_\alpha$ denote the finite-support product $\prod_{\beta<\alpha}\QQ_\beta$. Define $\cG_\alpha$ to be the finite-support product $\prod_{\beta<\alpha}\sG_\beta$ then $\cG_\alpha$ acts on $\RR_\alpha$ naturally. Similarly, let $\cF_\alpha$ be the normal filter of subgroups on $\cG_\alpha$ generated by the finite-support product of the filters $\sF_\beta$ for $\beta<\alpha$. Note that if $\alpha$ is a successor cardinal, then $|\RR_\alpha|<\alpha$, and therefore has $\alpha$-c.c., for an inaccessible cardinal $\alpha$ by a $\Delta$-system argument we get that $\RR_\alpha$ will also have the $\alpha$-.c.c., and in both case this property is preserved after forcing over $W$ with $\QQ_\alpha$. Therefore the conditions of the \autoref{lemma:yair product} hold, and so forcing with $\RR_\alpha$ over $W$ will preserve the distributivity of $\QQ_\alpha$.

We shall denote by $\HS_\alpha$ the hereditarily $\cF_\alpha$-symmetric $\RR_\alpha$-names. So in order to use the above fact on symmetric iterations, it is enough to prove the pointwise distributivity condition as stated in the fact. Define $\rho(\eta)$ to be the least $\kappa\in D$ such that $|W_\eta|<\kappa$. We prove by induction that for $\alpha\geq\rho(\eta)$, if $\dot x\in\HS_{\alpha+1}$ and $\dot x$ has rank $<\eta$, then $\dot x$ is essentially a name in $\HS_\alpha$. 

Suppose that $\dot x\in\HS_{\alpha+1}$ and $\dot x$ has rank $<\eta$. By the fact that $\RR_\alpha$ does not change the distributivity of $\QQ_\alpha$, and $\forces_{\RR_{\alpha+1}}|\dot x|<\check\alpha$, we get that there is a dense set of conditions $r$ such that for some $\dot x_r\in\HS_\alpha$, $r\forces_{\RR_{\alpha+1}}\dot x=\dot x_r$. We can find such $r$ such that $r(\alpha)$ is fixed by a $H$ group in $\sF_\alpha$. It follows that $\dot x_r$ is preserved by all the automorphisms in $\sym(\dot x)\cap H$. So $\dot x_r\in\HS_{\alpha+1}$. However, by the fact that $\dot x_r$ is actually an $\RR_\alpha$-name, any part of the automorphisms coming from $\sG_\alpha$ is irrelevant, so in fact $\dot x_r\in\HS_\alpha$ as wanted.

\section{Flatly stated consistency results}
From this work we can conclude the following statements are consistent with $\ZF+$``every successor cardinal is regular'', relative to the consistency of $\ZFC$. 
\begin{enumerate}
\item For every regular cardinal $\kappa$, the club filter of $\kappa$ is $\sigma$-incomplete.
\item For every regular cardinal $\kappa$, the club filter of $\kappa$ is $\kappa$-complete, but not normal. Therefore there are no normal uniform filters on any infinite ordinal.
\item For every regular cardinal $\kappa$, there is a $\sigma$-incomplete filter on $\kappa$ which has a $\kappa$-complete filter base.
\end{enumerate}
\section*{Acknowledgments}
The author would like to thank his advisor Menachem Magidor for his patient help and useful suggestions; to Yair Hayut for his continuous feedback, suggestions and general ideas for improvements; and Assaf Rinot for reading some early drafts of this manuscript and voicing some much needed critique about clarity of statements and expositions.
\bibliographystyle{amsplain}
\providecommand{\bysame}{\leavevmode\hbox to3em{\hrulefill}\thinspace}
\providecommand{\MR}{\relax\ifhmode\unskip\space\fi MR }
\providecommand{\MRhref}[2]{%
  \href{http://www.ams.org/mathscinet-getitem?mr=#1}{#2}
}
\providecommand{\href}[2]{#2}

\end{document}